\documentclass{amsart}
\usepackage{amsfonts}
\usepackage{amssymb}
\usepackage{ dsfont }
\usepackage{amsmath}
\usepackage{tikz}
\usepackage{graphicx}
\usepackage{ mathrsfs }
\newtheorem{mtheorem}{Theorem}
\newtheorem{theorem}{\textbf{Theorem}}[section]

\newtheorem{claim}[theorem]{\textbf{Claim}}

\newtheorem{lemma}[theorem]{\textbf{Lemma}}
\newtheorem{proposition}[theorem]{\textbf{Proposition}}
\newtheorem{remark}[theorem]{\textbf{Remark}}

\newtheorem*{question*}{\textbf{Question}}
\theoremstyle{plain}

\newcommand{\eqdef}{\stackrel{\scriptscriptstyle\rm def}{=}}

\newtheoremstyle{example}
  {}
  {}
  {}
  {}
  {\itshape}
  {.}
  {.5em}
  {\thmname{#1}\thmnumber{ #2}\thmnote{ (#3)}}
\theoremstyle{example}
\newtheorem{example}{Example}[section]

\numberwithin{equation}{section}
\begin{document}

\title{Markovian random iterations of maps}

\author[E. Matias]{Edgar Matias}
\address{Departamento de Matem\' atica, Universidade Federal da Bahia, Av. Adhemar de Barros s/n, 40170-110 Salvador, Brazil}
\email{edgar.matias@ufba.br}

\begin{abstract}
In this paper, we study Markovian random iterations of maps on standard measurable spaces. We establish a one-to-one correspondence between stationary measures and a certain class of invariant measures of a Markovian random iteration, extending a similar classical result of independent and identically distributed random iterations. As an application, we prove a local synchronization property for Markovian random iterations of homeomorphisms of the circle $S^{1}$.

\end{abstract}
\begin{thanks}{
This work was partially completed while the author
was supported by postdoctoral fellowships at Federal University of Rio de Janeiro and  ICMC-USP.
The author thanks to CAPES and Serrapilheira for the financial support.  The author warmly thanks Anna Zdunik and Tiago Pereira for their useful comments. 
The author is also very grateful to Katrin Gelfert and Tiago Pereira for the generous hospitality.
}\end{thanks}
\keywords{Random iteration of maps, synchronization, stationary measure, skew product, invariant measure. }
\subjclass[2010]{60J05,  37C40, 60G10.}

\maketitle

\section{Introduction}

Let $X=\{X_{n}\}$ be a homogeneous Markov chain moving through a measurable space $E$ and 
consider a family $\{f_{\alpha}\}_{\alpha\in E}$  of homeomorphisms of the circle $S^{1}$.  These two ingredients specify, for every $x$, 
a \emph{Markovian random iteration} given by 
$$
f_{X}^{n}(x)= f_{X_{n-1}}\circ \dots\circ f_{X_{0}}(x).
$$
For an independent and identically distributed (i.i.d.) sequence $\{X_{n}\}$, it was 
proved by Malicet in \cite{Malicet} that if the maps $f_{\alpha}$ do not have an invariant measure 
in common, then there is $\rho<1$ such that  for every $x\in S^{1}$, with probability 1, there is an interval $I$ containing $x$ such that 
\begin{equation}\label{localsync}
\mbox{diam}\, f_{X}^{n}(I)\leq \rho^{n} \quad \mbox{for every }\quad n\geq 0.
\end{equation}

In this paper, we extend this result to the case where $\{X_{n}\}$ is a Markov chain.
We will obtain this generalization as an application of a general result on Markovian random iterations of maps on standard measurable spaces (see Theorem \ref{one} below)  relating  \emph{stationary measures} and \emph{invariant measures}, 
which we start to describe now.

Now, we let $\{f_{\alpha}\}_{E}$ denote a family of measurable transformations of a standard measurable space $M$. A classical approach in the study of a random iteration is to consider a  dynamical system, the so-called \emph{skew product}, whose dynamical behavior is closely related to the random iteration. This allows to apply several results from the deterministic theory 
to the setting of random iterations. To this end, the Markov chain $X$ is taken to be in its canonical form, that is, $X_{n}$ is the natural projection on the product space $\Sigma=E^\mathbb{N}$, given by $X_{n}(\omega)=\omega_{n}$. Then, we consider the induced  \emph{skew product} defined by
$$
F(\omega,x)=(\sigma(\omega),f_{\omega_{0}}(x)),
$$
where $\sigma$ is the shift map on $\Sigma$. This map plays a key role in the study of  random iterations.

In the i.i.d. case the sequence of random variables  $f_{X}^{n}(x)$ is a homogeneous Markov chain with a well-defined transition probability (see Section \ref{Essephi}). Thus, for an 
i.i.d. random iteration, there are two classes of ``invariant measures'' that we can consider: the \emph{stationary measures} of this Markov chain and the \emph{invariant measures} of the induced skew product. An important fact about these two sets of probability measures is that there is a  one-to-one correspondence
between stationary measures and a certain class of invariant measures of the induced skew product. More precisely, a probability measure $\nu$ on $M$ is a stationary measure if and only if the product measure $\mathbb{P}\times \nu$ is invariant for the skew product,
 where $\mathbb{P}$ is the probability measure considered on the space $\Sigma$ for which $\{X_{n}\}$ is an i.i.d.  sequence.

There are various situations in the study of i.i.d. random iterations where 
this correspondence is invoked and explored in both directions, that is, 
studying  random iterations using results on skew products as well as
studying skew products using results from the general theory of Markov chains. A good 
example of this is the main theorem of \cite{Malicet}.
Therein, using the correspondence between stationary measures and invariant product measures, Malicet combined results on skew products (the invariance principle) and a kind of Krylov-Bogolyubov theorem for general Markov chains to prove the surprising \emph{local synchronization} property mentioned above in \eqref{localsync}. 

It is worth mentioning that this correspondence has found several applications to problems of a different nature. To mention some of them: 
in the theory of random matrices (e.g. the invariance principle in the i.i.d. case \cite{Led}
and the continuity of Lyapunov exponents of two-dimensional matrices \cite{Bocker}),  
in the theory of random pertubations \cite{VitorAraujo, Araujo,Hale} 
and also in the study of Poincar\' e recurrence theorems for random dynamical systems \cite{Rousseau}.

A natural question then would be whether there is such a correspondence in the Markovian case. The first obstacle we encounter is that, for Markovian random iterations, the sequence $f_{X}^{n}(x)$ is no longer a Markov chain. However, it turns out that 
the sequence of random variables
\begin{equation}\label{themarkovchain}
Z_{n}=(X_{n-1},f_{X}^{n}(x))
\end{equation}
 is a homogeneous Markov chain. Thus we can consider stationary measures related to 
a Markovian random iteration and ask whether there is  a set of invariant measures of the skew product that is in one-to-one correspondence to this set of stationary measures. In Theorem \ref{one}, we consider a Markovian random iteration of maps on a standard measurable space, and we establish a one-to-one correspondence  between stationary measures
 of  the Markov chain $Z_{n}$ and a certain class of invariant measures of the skew product (see Section \ref{Essephi} for the precise definition of this class). This bijection is given explicitly  using the disintegration of probability measures.
 Our result  generalizes  the classical correspondence of the i.i.d. setting,
providing a tool that could be useful to several branches of random dynamical systems.

 Moreover, under some additional assumptions, we show that a  stationary measure is ergodic if and only if its image under this bijection is an ergodic invariant measure, see Theorem \ref{ergodic}.
 In the i.i.d. case, this result was first proved by Kakutani \cite{Kakutani} 
under the hypothesis that the maps have a common fixed point. 
This hypothesis was removed by Ohno \cite{Ohno}. 
Let us observe that
for Markovian random iterations
 of finitely many maps, Markov chains as in \eqref{themarkovchain} were considered already in \cite{Recurrent, DiazMatias,  Edalat, Volk}. Therein conditions are provided ensuring the uniqueness or finiteness of ergodic stationary measures.

 Finally, as an application of the obtained results, we prove a local synchronization  property as in \eqref{localsync}
 for Markovian random iterations of homeomorphisms of $S^{1}$, see Theorem \ref{teocontraction}. This local synchronization  has several consequences in the study of the topological and the ergodic behavior of a random iteration. In the i.i.d. case for example,  in \cite{Malicet},  Malicet obtains the Antonov Theorem \cite{Antonov, KlepAntonov}  in a non-minimal context. Another example is  the central limit theorem for random iterations of finitely many homeomorphisms of  $S^{1}$, obtained in \cite{Annacentral}.  In a forthcoming paper, we will explore applications of the local
 synchronization property obtained in Theorem \ref{teocontraction}  for Markovian random iterations.

\subsection*{Organization of the paper}

In Section \ref{mainn} we state precisely the main definitions and results of this work. In Section \ref{sectiontone} we characterize stationary measures and invariant measures using the disintegration of measures, and we prove Theorem \ref{one}. Theorem \ref{ergodic}  is proved in Section \ref{sectionergodic}. In Section \ref{sectionlocsinc} we apply Theorem \ref{one} to deduce Theorem \ref{teocontraction}. 

\section{Statement of results}\label{mainn}

\subsection{Markovian random iterations}\label{prelimi}
Let $(E,\mathscr{E})$ be a standard measurable space
 and consider a transition probability 
 $p\colon E\times \mathscr{E}\to [0,1]$, that is,  
 for every $\alpha\in E$ the mapping $A\mapsto p(\alpha,A)$ is a probability measure on $E$ and 
 for every $A\in \mathscr{E}$ the mapping $\alpha \mapsto p(\alpha,A)$ is measurable with respect
  to $\mathscr{E}$. 

Let $\Sigma$ be the product space  $E^{\mathbb{N}}$ endowed with the product $\sigma$-algebra $\mathscr{F}=\mathscr{E}^{\mathbb{N}}$ and $m$ be a probability measure on $E$. Associated with the pair $(p,m)$  there is a unique probability measure $\mathbb{P}$ on $\Sigma$, called \emph{Markov measure}, for which the sequence of natural projections $X_{n}(\omega)=\omega_{n}$ is a stationary Markov 
chain on the probability space $(\Sigma, \mathscr{F}, \mathbb{P})$ with transition  probability $p$ and starting measure $m$, see \cite{Revuz}.
 Recall that a probability 
 measure $m$ on $E$ is called a \emph{stationary measure} of the transition probability  $p$ if 
 $$
 m(A)=\int p(\alpha,A)\, dm(\alpha),
 \quad
 \mbox{for every $A\in \mathscr{E}$.}
 $$
Throughout this paper, we assume that the transition probability $p$ has a unique stationary measure $m$, and $\mathbb{P}$ stands for the Markov measure associated with $(p,m)$.

Let $(M, \mathscr{B})$ be a standard measurable space.
 Let $f\colon E\times M\to M$ be a measurable map and denote by $f_{\alpha}$ the map 
 $f_{\alpha}(x)=f(\alpha,x)$. Then,
the map
 $\varphi\colon \mathbb{N}\times \Sigma\times M\to M$ defined by 
$$ 
\varphi(n,\omega,x)\eqdef f_{\omega_{n-1}}\circ\dots \circ f_{\omega_{0}}(x)\eqdef f_{\omega}^{n}(x)
$$
is called a \emph{Markovian random iteration of maps}. The term Markovian means that in the study of this map under a probabilistic point of view, we consider 
the space $\Sigma$ endowed with a Markov measure. We often refer to $\varphi$ as a Markovian random iteration of maps associated with the pair $(p,m)$ and the map $f$.

For every $x\in M$,
 the sequence of random variables 
$Z_{n}(\omega)=(\omega_{n-1}, f_{\omega}^{n}(x))$ is 
a Markov chain with range on the space $E\times M$  with transition probability given by 
\begin{equation}\label{transition}
 \hat p((\alpha,x),B)\eqdef \int \mathds{1}_{B}(\beta, f_{\beta}(x))\,p(\alpha,d\beta)
 \qquad\mbox{for every}\,\,\,\, B\in \mathscr{E}\otimes \mathscr{B},
\end{equation}
where $\mathds{1}_{B}$ indicates the characteristic function of $B$. A stationary measure 
for this  transition probability is called a \emph{stationary measure for the Markovian random iteration} 
$\varphi$ and the set of stationary measures of $\varphi$ is denoted by $\mathcal{S}(\varphi)$.

The \emph{skew product} induced by $\varphi$ is the map 
 $F\colon \Sigma\times M\to \Sigma\times M$   defined by 
\begin{equation}\label{defiskewproduct}
F(\omega,x)\eqdef(\sigma(\omega),f_{\omega_{0}}(x)),
\end{equation}
where $\sigma$ is the shift map on $\Sigma$. We also say that $F$ is the skew product induced by the map $f$.
We say that a probability measure $\hat \mu$ is $F$-\emph{invariant} if $F_{*}\hat\mu=\hat\mu$, where $*$ denotes the push-forward of a measure.

\subsection{Correspondence between stationary measures and invariant measures}\label{Essephi}
Let $(Y, \mathscr{Y})$ be a measurable space and consider a probability measure $\eta$ on the product space $Y\times M$. Let $\Pi_{1}\colon Y\times M\to Y$ be the natural projection on the first factor given by $\Pi_{1}(y,x)=y$. Then there is a unique (a.e.) family of probability measures $\{\eta_{y}\colon y\in Y\}$ on the space $M$ such that for every measurable rectangle $A\times B$, the map $y\mapsto \eta_{y}(B)$ is measurable and 
\begin{equation}\label{disintegralogo}
\eta(A\times B)=\int_{A} \eta_{y}(B)\, d\,\Pi_{1*}\eta.
\end{equation}
The family  $\{\eta_{y}\colon y\in Y\}$ is called the \emph{disintegration of $\eta$} with respect to $\Pi_{1*}\eta$,
see Rokhlin \cite{Rolin}. The probability measure $(\Pi_{1})_{*}\eta$ on $Y$ is called the \emph{first marginal} of $\eta$. Reciprocally, let $\{\eta_{y}\colon y\in Y\}$ be a family of probability measures  on the space $M$ such that for every measurable set $B\subset M$ the map $y\mapsto \eta_{y}(B)$ is measurable. Then for every probability measure 
$\tau$ on $Y$ there is unique probability measure $\eta$ on the product space $Y\times M$
 with first marginal $\tau$ whose disintegration with respect to $\tau$ is given by this family. Indeed, we define $\eta$ as in \eqref{disintegralogo} with $\tau$ in the place of $\Pi_{1*}\eta$.

Let $\varphi$ be the Markovian random iteration associated with the pair $(p,m)$ and the map $f$.
 We observe that if $\nu$ is a stationary measure of $\varphi$, then the first marginal of $\nu$ 
is a stationary measure of the transition probability $p$ (see Proposition \ref{invariant} below). 
Therefore, since we are assuming that $p$ has a unique stationary measure $m$, we conclude that $m$ is the first marginal of $\nu$. In particular, we can consider the disintegration $\{\nu_{\alpha}\colon \alpha \in E\}$ of $\nu$ with respect to $m$.  
 
 Let $F$ be the skew product induced by $\varphi$.
An $F$-invariant measure  with first marginal $\mathbb{P}$ is called a $\varphi$-\emph{invariant measure}.  Let $I_{0}(\varphi)$ be the set of $\varphi$-invariant measures $\hat \mu$  for which there is a family $\{\mu_{\alpha}\colon \alpha\in E\}$ of probability measures on $M$ such that 
 $$
 \hat \mu_{\omega}=\mu_{\omega_{0}}
 $$ 
 for every $\omega$, where the family $\{\hat \mu_{\omega}\colon \omega\in \Sigma\}$ is the disintegration of $\hat \mu$ with respect to $\mathbb{P}$. In other words,  $I_{0}(\varphi)$ consists of $\varphi$-invariant measures whose disintegration depends only on the zeroth coordinate of $\omega$.

Our first main theorem state a one-to-one correspondence between 
the sets $I_{0}(\varphi)$ and $S(\varphi)$. This bijection is given explicitly,  and the proof relies on the characterization of stationary measures and 
invariant measures of $I_{0}(\varphi)$ using  disintegration of measures (see Propositions \ref{invariant} and \ref{stati}). To this end, we need the following definition. 
We say that a transition  probability $q$ on $E$ is in \emph{duality} with the transition probability $p$ relative to the stationary measure $m$ if for every measurable sets $A$ and $B$ we have
$$
\int_{B} p(\alpha,A)\,dm(\alpha)=\int_{A} q(\beta,B)\,dm(\beta).
$$
On standard measurable spaces, there is an essentially unique  transition probability in duality 
with $p$ relative to $m$, see \cite[Lemma 4.7 and Theorem $4.9$]{Revuz}.

\begin{mtheorem}\label{one}   Let $\varphi$ be a Markovian random iteration associated with the pair $(p,m)$ and the map $f$. Let $q$ be the transition probability in duality with $p$ relative to $m$. Then, 
there is a one-to-one correspondence between the sets $S(\varphi)$ and $I_{0}(\varphi)$
given by 
$$
\hat \mu\in I_{0}(\varphi)\mapsto \nu\in \mathcal{S(\varphi)},\quad\, \nu_{\alpha}=
f_{\alpha*}\mu_{\alpha}
$$
and 
$$
\nu \in \mathcal{S(\varphi)}\mapsto \hat \mu \in
 I_{0}(\varphi),\quad \mu_{\alpha}=\int \nu_{\beta}\, q(\alpha,d\beta).
$$
\end{mtheorem}

This correspondence is classical in the i.i.d. case. Let us explain this a bit more precisely. If $\mathbb{P}$ is the product measure $m^\mathbb{N}$, then for every $x\in M$, the sequence of random variables
$
\omega \mapsto f^{n}_{\omega}(x)
$
is a homogeneous Markov chain whose transition probability is given by 
$$
p(x,A)=\int \mathds{1}_{A}(f_{\alpha}(x))\, dm(x) \qquad\mbox{for every}\,\,\, A\in \mathscr{E}. 
$$
In \cite{Ohno}, Ohno has shown  that a probability measure $\nu$ is a stationary measure of $p$ if and only if the product measure $\mathbb{P} \times \nu$ is $F$-invariant measure.

 Moreover, Ohno proved that a stationary measure $\nu$ is ergodic if and only if the product measure
 $\mathbb{P} \times \nu$ is an ergodic  $F$-invariant measure. See \cite{Kifer, vianalivro} for different proofs of this fact. 
Assuming a certain condition on the pair $(p,m)$, we obtain the same result for Markovian random iterations (see Theorem \ref{ergodic} below). To introduce this condition denote by $p_{\alpha}$ the probability measure $p_{\alpha}(A)=p(\alpha,A)$. We say that  the pair  $(p,m)$ is \emph{bounded} if  $p_{\alpha}$ is absolutely  continuous with respect to $m$ for every $\alpha$  and there is a constant $C$ with $0<C\leq 1$ such that
\begin{equation}\label{boundedd}
C\leq \frac{dp_{\alpha}}{dm}(\beta)\leq C^{-1}
\end{equation}
for every $\beta$ and $\alpha$ in $E$. Moreover, we say  that a Markovian random iteration $\varphi$ associated with $(p,m)$ is \emph{bounded} if the pair $(p,m)$ is bounded.

Denoting by $\Xi\colon S(\varphi)\to I_{0}(\varphi)$ the bijection of Theorem \ref{one}, we are now ready to state our second main result.
\begin{mtheorem}\label{ergodic}
Let $\varphi$ be a bounded Markovian random iteration. Then, a stationary measure $\nu\in S(\varphi)$ is ergodic if and  only if the the $\varphi$-invariant measure   $\Xi(\nu)$ is  ergodic.
\end{mtheorem}

We now present some examples of bounded transition probabilities.
\begin{example}[Random iteration driven by random walks]
Consider a sequence $\{W_{n}\}_{n\geq 0}$ of i.i.d. random variables taking values in $S^{1}=\mathbb{R}\setminus\mathbb{Z}$ whose common distribution is the Lebesgue measure $Leb$ of $S^{1}$. Then the sequence 
$$
X_{n}=W_{1}+\dots+W_{n}
$$
is a Markov chain with transition probability given by $p(x, A)=Leb(A)$. Note that
$Leb$ is a stationary measure of this Markov chain. In particular, the Radon-Nikodym derivative of $p_{x}$ with respect to the Lebesgue measure is equal to one. Hence, the pair $(p,Leb)$ is bounded.

More generally,
let $G$ be a compact topological group endowed with its Borel $\sigma$-algebra $\mathscr{G}$. Recall that a \emph{left random walk} on $G$ with law $\mu$ is a Markov chain $\{X_{n}\}$ with transition probability given by 
$$
p(g,A)=\mu\ast\delta_{g}(A) \qquad \mbox{for every}\,\,\, A\in \mathscr{G},
$$
where $\ast$ denotes the convolution of measures.
We observe that the
Haar measure $m$ of $G$ is a stationary measure of $p$. Let us assume that $\mu$ is absolutely continuous with respect to  $m$ and let $f$ denote the Radon-Nikodym derivative $\frac{d\mu}{dm}$. We claim that
$$
\frac{dp_{g}}{dm}(h)=f(hg^{-1}). 
$$
Indeed,  we have
\[ 
\begin{split}
p_{g}(A)=\mu\ast\delta_{g}(A)&=\int \mathds{1}_{A}(hx)\,d\mu(h)\,d\delta_{g}(x)\\
&=\int \mathds{1}_{A}(hg)\,d\mu(h)\\
&=\int f(h) \mathds{1}_{A}(hg)\, dm(h)\\
&=\int f(hg^{-1})\mathds{1}_{A}(h)\, dm(h)=\int_{A}f(hg^{-1})\, dm(h),
\end{split}
\] 
where in the last equality we use that the Haar measure on a compact topological group is invariant for left and right translations.

In particular, if there is $C>0$ such that the Radon-Nikodym derivative $f$ satisfies 
$$
C\leq f(g)\leq C^{-1}
$$
for  every $g\in G$, then the pair $(p, m)$ is bounded.
\end{example}

\begin{example}[Random iteration of finitely many maps]
Consider a transition matrix $P=(p_{ij})$ with positive entries. Then it follows from Perron-Frobenius theorem that the matrix  $P$ has a unique stationary probability vector $p=(p_{i})$ with positive entries. Thus,  condition \eqref{boundedd} is verified for the pair $(P,p)$.
\end{example}

\subsection{Local synchronization}\label{localsynchronization}
Finally, we combine the correspondence between $I_{0}(\varphi)$ and $S(\varphi)$ with the invariance principle  \cite[Theorem F]{Malicet} to obtain a local synchronization 
property for Markovian random iterations of homeomorphisms of the circle $S^{1}$.

In what follows in this section, $p$ is a transition probability on a compact metric space $E$ (endowed with its Borel $\sigma$-algebra) having a unique stationary measure $m$, the pair $(p,m)$ is bounded, and we assume that
the mapping $\alpha\mapsto p(\alpha,\cdot)$ is continuous in the weak$\star$-topology.

\begin{mtheorem}\label{teocontraction}
Let $\varphi$ be a Markovian random iteration associated with the pair $(p,m)$ and a measurable map $f\colon E\times S^{1}\to S^{1}$. Assume that the map $f_{\alpha}$ is a homeomorphism of $S^{1}$  for every $\alpha$  and there is no probability measure $\mu$ such that $f_{\alpha*}\mu=\mu$ for $m$-almost every $\alpha$.
 Then there is $\rho<1$ such that for any $x\in S^{1}$, for $\mathbb{P}$-almost every $\omega\in \Sigma$,
 there is a neighbourhood $I$ of $x$ such that 
$$
\mathrm{diam}\, f_{\omega}^{n}(I)\leq \rho^{n} \quad \mbox{for every} \quad n\geq 0.
$$
\end{mtheorem}
Theorem \ref{teocontraction} is
new even in the setting of random iterations of $C^{1}$ diffeomorphisms and extends Theorem $\mbox{A}$ of  \cite{Malicet} for Markovian random iterations.
Let us observe that this synchronization phenomenon for Markovian random iterations of finitely many maps on compact subsets of a finite dimensional Euclidean space 
was studied already in \cite{DiazMatias}.  Assuming a purely topological condition on the maps, called \emph{splitting condition}, the authors proved a contraction exponentially fast of the whole space under the action of the random iteration. 
However,  in the circle $S^{1}$, this splitting condition is never satisfied.  

\section{Proof of Theorem \ref{one}}
\label{sectiontone}
Let $(M,\mathscr{B})$ and $(E,\mathscr{E})$ be standard measurable spaces, $f\colon E\times M\to M$ be a measurable map and $p$ be a transition probability on $E$ with a unique stationary measure $m$.
Let $\varphi$ be the Markovian random iteration associated with the pair $(p,m)$ and the map $f$, as defined in Section \ref{prelimi}.
We start by giving a characterization of the elements of $S(\varphi)$ and $I_{0}(\varphi)$ using the disintegration of measures. In what follows, $q$ denotes the transition probability in duality with $p$ relative to $m$.

\subsection{Stationary measures}

Let $\hat  p$ be  the transition probability  on $E\times M$ induced by $\varphi$ as defined in \eqref{transition}.
The {\emph{Markov operator}}  $\mathcal{P}$ induced by $\hat p$ is defined as follows: given a probability measure 
 $\nu$ on $E\times M$ we define $\mathcal{P}\nu$  by
$$
 \mathcal{P}\nu(B)\eqdef \int  \hat{p}((\alpha,x),B)\,d\nu(\alpha,x) \quad \mbox{for}\,\, B\in \mathscr{E}\otimes \mathscr{B}.
 $$
  Note that by definition, a stationary measure of $\varphi$ is a fixed point of $\mathcal{P}$.
 
The following proposition is a useful characterization of stationary measures using  the disintegration of measures:

\begin{proposition}\label{invariant}
 Let $\nu$ be a probability 
measure on $E\times M$. The following statements hold: 
\begin{enumerate}
\item[(i)]
If $\nu$ is a stationary measure of $\varphi$, then $\nu$ has first marginal $m$.  

\item[(ii)] If $\nu$ has first marginal $m$, then $\mathcal{P}\nu$ has first marginal $m$ and its disintegration 
with respect to $m$ is given by 
$$
(\mathcal{P}\nu)_{\alpha}=\int f_{\alpha*}\nu_{\beta}\,q(\alpha,d\beta)
$$
for $m$-almost every $\alpha$.
\item[(iii)]
If $\nu$ has first marginal $m$, then $\nu$ is a stationary measure if and only if
$$
\nu_{\alpha}=\int f_{\alpha*}\nu_{\beta}\, q(\alpha,d\beta)
$$
for $m$-almost every $\alpha$.

\end{enumerate}

\end{proposition}
\begin{proof}

We stat by proving item (i). Assume that $\nu$ is a stationary measure and let $m'=(\Pi_{1})_{*}\nu$. Then given a measurable set $A\in \mathscr{E}$, we have 
\[
\begin{split}
m'(A)=\nu (A\times M)=\mathcal{P}\nu(A)&=\int \int \mathds{1}_{A\times M}(\beta, f_{\beta}(x))\,p(\alpha,d\beta)\,d\nu(\alpha,x)
\\
&=\int \int \mathds{1}_{A}(\beta)\,p(\alpha,d\beta)\,d\nu(\alpha,x)\\
&=\int p(\alpha,A) \,d\nu(\alpha,x)\\
&=\int \int p(\alpha,A)\, \nu_{\alpha}(x)dm'(\alpha)
=\int p(\alpha,A)\, dm'(\alpha).
\end{split}
\]
Hence $m'$ is a stationary measure of $p$. Since we are assuming that $m$ is the unique stationary measure of $p$, we conclude that $m'=m$. 

To prove item (ii),
assume that $\nu$ has first marginal $m$. Then for every measurable set $A\in\mathscr{E}$, we have
\[
\begin{split}
\mathcal{P}\nu(A\times M)&=\int \int \mathds{1}_{A\times M}(\beta, f_{\beta}(x))\,p(\alpha,d\beta)\,d\nu(\alpha,x)\\
&=\int \int p(\alpha,A)\, \nu_{\alpha}(x)dm(\alpha)=\int \int p(\alpha,A)\, \nu_{\alpha}(x)dm(\alpha)=m(A),
\end{split}
\]
which implies that the first marginal of $\mathcal{P}\nu$ is $m$.

We now compute the disintegration of  $\mathcal{P}\nu$ with respect to $m$. First, we have 
\begin{equation}\label{serra}
\begin{split}
\mathcal{P}\nu(A\times B)&=
\int \int \mathds{1}_{A}(\beta)\mathds{1}_{B}(f_{\beta}(x))\,p(\alpha,d\beta)\,d\nu(\alpha,x)\\
&=\int\int \int \mathds{1}_{A}(\beta)\mathds{1}_{B}(f_{\beta}(x))\,p(\alpha,d\beta)\,d\nu_{\alpha}(x)dm(\alpha)
\\
&=\int\int \int \mathds{1}_{A}(\beta)\mathds{1}_{B}(f_{\beta}(x))\,d\nu_{\alpha}(x)\,p(\alpha,d\beta)dm(\alpha).
\end{split}
\end{equation}
Next, we need the following well-known property of transition probabilities in duality: for every measurable map $\kappa\colon E\times E\to [0,\infty)$ we have
\begin{equation}\label{permuta}
\int \kappa(\alpha,\beta)\, p(\alpha,d\beta)\, dm(\alpha)=\int \kappa(\beta,\alpha)\,
 q(\alpha,d\beta)\, dm(\alpha).
\end{equation}
Applying this property to the right term of the last equality in \eqref{serra} we obtain:
\[
\begin{split}
\mathcal{P}\nu(A\times B)&=\int\int \int \mathds{1}_{A}(\alpha)\mathds{1}_{B}(f_{\alpha}(x))\,d\nu_{\beta}(x)\,q(\alpha,d\beta)dm(\alpha)\\
&=\int_{A}\int f_{\alpha*}\nu_{\beta}(B)\,q(\alpha,d\beta)dm(\alpha).
\end{split}
\]
This implies that $(\mathcal{P}\nu)_{\alpha}=\int f_{\alpha*}\nu_{\beta}\, q(\alpha,d\beta)$ for 
$m$-almost every $\alpha$.

Finally, item (iii) follows immediately from items (i) and (ii).

\end{proof}

\subsection{Invariant measures}

The following proposition is a useful characterization of the elements of 
$I_{0}(\varphi)$, recall the definition of this set in Section \ref{Essephi}. In what follows $\mathbb{P}$ denotes the Markov measure on $\Sigma=E^{\mathbb{N}}$ associated with the pair $(p,m)$, as defined in Section \ref{prelimi}. 
\begin{proposition}\label{stati}
Let $\hat \mu$ be a probability measure on $\Sigma\times M$ with first marginal $\mathbb{P}$ and assume that there is a family $\{\mu_{\alpha}\colon \alpha\in E\}$ of probability measures on $M$ such that 
 $
 \hat \mu_{\omega}=\mu_{\omega_{0}}
 $
 for every $\omega$. Then $\hat \mu$ is $\varphi$-invariant if and only if 
$$
\mu_{\alpha}=\int f_{\beta*}\mu_{\beta}\, q(\alpha,d\beta)
$$
for $m$-almost every $\alpha$.
\end{proposition}

\begin{remark}\emph{
 In \cite{Malheiro}, Malheiro and Viana introduced a notion of stationary measure for 
Markovian random products of finitely many matrices $A_{1},\dots, A_{k}$. Let us observe that their notion is different from the one considered in this paper. Indeed, let $P=(p_{ij})$ be a $k\times k$  transition matrix with a unique stationary probability vector $p=(p_{1},\dots, p_{k})$. A stationary measure in \cite{Malheiro} is defined to be a $k$-tuple $(\mu_{1},\dots, \mu_{k})$ 
of probability measures such that 
$$
\mu_{i}=\sum_{j=1}^{k} \frac{p_{j}}{p_{i}}p_{ji}A_{j*}\mu_{j}.
$$
It turns out that the transition matrix $Q=(q_{ij})$ in duality with $P$ relative to $p$ is given by 
$$
q_{ij}=\frac{p_{j}}{p_{i}}p_{ji}.
$$
Hence, Proposition \ref{stati} implies that the $k$-tuple $(\mu_{1},\dots, \mu_{k})$ 
is the disintegration of an $F$-invariant measure whose disintegration depends only on the zeroth coordinate. Therefore, in view of Proposition \ref{invariant} and Proposition \ref{stati}, we conclude  that the notion of stationary measures considered in this paper and the one considered in \cite{Malheiro} are different.} 
\end{remark}
\begin{proof}[Proof of Proposition \ref{stati} ]
Let $F$ be the skew product induced by $\varphi$ as defined in \eqref{defiskewproduct}. Let us compute the disintegration of $F_{*}\hat \mu$. First, note that 
\begin{equation}\label{delta1}
F_{*}\hat \mu(A\times B)=\hat \mu(F^{-1}(A\times B))=\int \mathds{1}_{A}(\sigma(\omega))
\mu_{\omega_{0}}(f_{\omega_{0}}^{-1}(B))\, d\mathbb{P}(\omega).
\end{equation}
Now, we need the following lemma:

\begin{lemma}\label{olharaprova}
Given  measurable maps $u\colon \Sigma\to \mathbb[0,\infty)$ and $g\colon E\to \mathbb[0,\infty)$, we have
$$
\int u(\sigma(\omega))g(\omega_{0})\,d \mathbb{P}(\omega)=\int\left( u(\omega) \int g(\beta)\, q(\omega_{0},d\beta)\right)\, d\mathbb{P}(\omega).
$$
\end{lemma}
\begin{proof} We start by recalling the Markov property  (see Revuz \cite[Proposition $2.13$]{Revuz}). Let $\mathbb{P}_{\alpha}$ be the Markov measure associated with $(p,\delta_{\alpha})$, where $\delta_{\alpha}$ is the Dirac measure centred on $\alpha\in E$. Then, the Markov property says that 
\begin{equation}\label{mp}
\int g(\omega_{0}) u(\sigma(\omega))\, d\mathbb{P}(\omega)
=\int g(\omega_{0})\mathbb{E}_{\omega_{1}}(u)\,d\mathbb{P}(\omega),
\end{equation}
where $\mathbb{E}_{\alpha}$ denotes the expectation with respect the Markov measure
$\mathbb{P}_{\alpha}$.

Using \eqref{mp}  and the definition of $\mathbb{P}$ we have 

\[
\begin{split}
\int g(\omega_{0}) u(\sigma(\omega))\, d\mathbb{P}(\omega)
&=\int g(\omega_{0}) 
\left(\int u(\vartheta)\, d\mathbb{P}_{\omega_{1}}(\vartheta) \right)\, d\mathbb{P}(\omega)\\
&=\iint g(\omega_{0}) 
\left(\int u(\vartheta)\, d\mathbb{P}_{\omega_{1}}(\vartheta) \right) p(\omega_{0},d\omega_{1})\, dm(\omega_{0}).
\end{split}
\]
 Applying  \eqref{permuta} to the right term of the last equality above we obtain
\[
\begin{split}
\int g(\omega_{0}) u(\sigma(\omega))\, d\mathbb{P}(\omega)
&=\iint
\left(\int g(\omega_{1})  u(\vartheta)\, d\mathbb{P}_{\omega_{0}}(\vartheta) \right) q(\omega_{0},d\omega_{1})\, dm(\omega_{0})\\
&=\int \int \int g(\omega_{1})  u(\vartheta)\, q(\omega_{0},d\omega_{1})\,
 d\mathbb{P}_{\omega_{0}}(\vartheta)\,\, dm(\omega_{0}),
\end{split}
\] 
where the second equality  follows from Fubini's theorem.

 We recall that every Markov measure associated with the transition probability $p$ can be obtained by the  family of Markov measures  $\{\mathbb{P}_{\alpha}\}_{\alpha\in E}$. That is, if $\mathbb{P}$ a is Markov measure induced by $(p,m)$, then
 $$
 \mathbb{P}(C)=\int \mathbb{P}_{\alpha}(C)\, dm(\alpha)\qquad \mbox{for every}\,\,\, C\in \mathscr{F}.
 $$ 
In particular, this implies that 
\[
\begin{split}
\int g(\omega_{0}) u(\sigma(\omega))\, d\mathbb{P}(\omega)&=
\int \int u(\vartheta)\left(\int g(\omega_{1})\, q(\omega_{0},d\omega_{1})  \right) d\mathbb{P}_{\omega_{0}}(\vartheta)\,\, dm(\omega_{0})\\
 &=\int\left( u(\omega) \int g(\beta)\, q(\omega_{0},d\beta)\right)\, d\mathbb{P}(\omega).
\end{split}
\]
  
%
  
%

The proof of the lemma is now complete.

\end{proof}

To conclude the proof of the proposition, we apply Lemma \ref{olharaprova} in \eqref{delta1} to obtain 
\[
\begin{split}
F_{*}\hat \mu(A\times B)&=\int_{A}\int
f_{\beta*}\mu_{\beta}(B)\,q(\omega_{0},d\beta)\, d\mathbb{P}(\omega), 
\end{split}
\]
which implies that 
$$
(F_{*}\hat \mu)_{\omega}=\int
f_{\beta*}\mu_{\beta}\,q(\omega_{0},d\beta)
$$
 for $\mathbb{P}$-almost every $\omega$. Therefore,  $\hat \mu$ is $F$-invariant if and only if 
 $$
 \mu_{\omega_{0}}=\int
f_{\beta*}\mu_{\beta}\,q(\omega_{0},d\beta)
 $$ 
 for $\mathbb{P}$-almost every $\omega$, or equivalently, for $m$-almost every $\omega_{0}$.
\end{proof}

\subsection{Proof of Theorem \ref{one}}

Let $\Theta$ be the map that associates to each $\hat \mu\in I_{0}(\varphi)$ the probability measure on $E\times M$ given by
$$
\Theta(\hat\mu)=\nu, \qquad \mbox{where} \quad \nu_{\alpha}=
f_{\alpha*}\mu_{\alpha}.
$$
Let $\Xi$ be the map that associates to each  $\nu\in S(\varphi)$ the probability measure on $\Sigma\times M$  given by
$$
\Xi(\nu)=\hat \mu, \qquad\mbox{where}\quad \hat \mu_{\omega}=\int \nu_{\beta}\, q(\omega_{0},d\beta).
$$
We claim that 
 $\Theta$ takes  $I_{0}(\varphi)$ into  
$S(\varphi)$ and $\Xi$ takes $S(\varphi)$ into $I_{0}(\varphi)$. We start by proving
the first claim. To this end, let $\hat\mu\in I_{0}(\varphi)$. It follows from the definition of $\nu=\Theta(\hat \mu)$  that, for every $\alpha$ and $\beta$, we have   
$$
f_{\alpha*}\nu_{\beta}=f_{\alpha*}f_{\beta*}\mu_{\beta}.
$$
 Now, recall from Proposition \ref{stati} 
that $\mu_{\alpha}=\int f_{\beta*}\mu_{\beta}\, q(\alpha, d\beta)$ for $m$-almost every $\alpha$. Hence
\[ 
\begin{split}
\int f_{\alpha*}\nu_{\beta}\, q(\alpha,\beta)&=\int f_{\alpha*}f_{\beta*}\mu_{\beta}\, q(\alpha,d\beta)
\\
&=f_{\alpha*}\int f_{\beta*}\mu_{\beta}\, q(\alpha,d\beta)=f_{\alpha*}\mu_{\alpha}=\nu_{\alpha}
\end{split}
\]
for $m$-almost every $\alpha$. We conclude then, from Proposition \ref{invariant}, that $\nu$ is a stationary measure. 

We now prove that $\Xi(S(\varphi))\subset I_{0}(\varphi)$. To see this, take $\nu\in S(\varphi)$ and let $\{\mu_{\beta}\colon\beta\in E\}$ be the family of probability measures defined by 
\begin{equation}\label{jtnoarx}
\mu_{\beta}=\int \nu_{\alpha} \,q(\beta,d\alpha).
\end{equation}
By definition,  $\hat \mu =\Xi(\nu)$ is the unique probability measure on $\Sigma\times M$ whose disintegration $\{\hat \mu_{\omega}\colon \omega\in \Sigma\}$ with respect to $\mathbb{P}$ is given by $\hat \mu_{\omega}=\mu_{\omega_{0}}$.

 It follows from Proposition \ref{invariant} that
\begin{equation}\label{jtnoarx2}
f_{\beta*}\mu_{\beta}=f_{\beta*}\left(\int \nu_{\alpha}\, q(\beta,d\alpha)\right)
=\int f_{\beta*}\nu_{\alpha}\, q(\beta,d\alpha)=\nu_{\beta}
\end{equation}
for $m$-almost every $\beta$.
Combining \eqref{jtnoarx} with \eqref{jtnoarx2}, we obtain
$$
\int f_{\beta*}\mu_{\beta}\, q(\alpha,d\beta)=\int \nu_{\beta}\, q(\alpha,d\beta)
=\mu_{\alpha}. 
$$
Therefore, Proposition \ref{stati} implies that $\hat \mu$ is  a $\varphi$-invariant measure
whose disintegration depends only on the zeroth coordinate.

We now conclude the proof of Theorem \ref{one}. Note that to this end, it is sufficient to prove that $\Theta\circ \Xi(\nu)=\nu$ for every $\nu \in S(\varphi)$ and $\Xi\circ \Theta(\hat\mu)=\hat \mu$ for every $\hat \mu\in I_{0}(\varphi)$.

We start by proving that $\Theta\circ \Xi(\nu)=\nu$. Let $\hat \mu=\Xi(\nu)$ and consider the family $\{\mu_{\alpha}\colon\alpha\in E\}$ of probability measures defined by 
$$
\mu_{\alpha}=\int \nu_{\beta} \,q(\alpha,d\beta).
$$ 
By definition, $\Theta(\hat \mu)$ is the probability measure $\nu'$ whose disintegration
with respect to $m$ is given by 
$
\nu'_{\alpha}=f_{\alpha*}\mu_{\alpha}.
$  
Then, we have
$$
\nu'_{\alpha}=f_{\alpha*}\mu_{\alpha}=f_{\alpha*}\int \nu_{\beta} \,q(\alpha,d\beta)=
\int f_{\alpha*} \nu_{\beta} \,q(\alpha,d\beta)=\nu_{\alpha}
$$
for $m$-almost every $\alpha$, where in the last equality we use Proposition \ref{invariant}.  Thus we conclude that $\nu=\nu'$, which means that $\Theta\circ \Xi(\nu)=\nu$.

We now prove that $\Xi\circ \Theta(\hat\mu)=\hat \mu$. 
Let $\{\mu_{\alpha}\colon \alpha\in E\}$ be the family of probability measures for which  
$\hat \mu_{\omega}=\mu_{\omega_{0}}$ for $\mathbb{P}$-almost every $\omega$. Then, 
by definition  the stationary measure $\nu=\Theta(\hat \mu)$  is given by
 $\nu_{\alpha}=f_{\alpha*}\mu_{\alpha}$. Hence, the $\varphi$-invariant measure $\hat \mu'= \Xi(\nu)$ is given by 
 $$
 \hat \mu'_{\omega}=\int \nu_{\alpha}\,q(\omega_{0},d\alpha)=
 \int f_{\alpha*}\mu_{\alpha}\, q(\omega_{0},d\alpha)=\mu_{\omega_{0}},
 $$
 where is the last equality we use Proposition \ref{stati}.
 This implies that $\Xi\circ \Theta(\hat\mu)=\hat \mu$.

  The proof of the theorem  is now complete.

  \hfill \qed

\section{Proof of Theorem \ref{ergodic}}\label{sectionergodic}
We start by proving a technical result on bounded Markovian random iterations.
Let $(M,\mathscr{B})$ and $(E,\mathscr{E})$ be standard measurable spaces, $f\colon E\times M\to M$ be a measurable map and $p$ be  a transition probability on $E$ with a unique stationary measure $m$. In what follows the pair $(p,m)$ is bounded, that is,  $p_{\alpha}$ is absolutely continuous with respect to $m$ for every $\alpha$ and
there is a constant  $C$ with $0<C\leq 1$ such that 
\begin{equation}\label{faci}
C\leq \frac{dp_{\alpha}}{dm}\leq C^{-1}\qquad \mbox{for every}\,\,\, \alpha.
 \end{equation} 
  Throughout this section $C$ denotes a constant satisfying condition  \eqref{faci}. We denote by $\mathscr{F}_{n}$  the $\sigma$-algebra on $\Sigma$ generated by the canonical projections 
$\omega\mapsto\omega_{i}$, $i=0,\dots, n$. The following lemma is a general result on Markovian random iterations, also needed in the proof of Theorem \ref{teocontraction}.
In what follows, $\mathbb{E}(\cdot|\cdot)$ denotes the conditional expectation of a measurable map with respect to a $\sigma$-algebra and the Markov measure $\mathbb{P}$.
\begin{lemma}\label{difmarkocase}Let $F$ be the skew product induced by the map $f$.
Let $h\colon \Sigma\to [0,\infty)$ be a measurable map such that
$h(\omega,x)\geq h(F(\omega,x))$ for every $(\omega,x)\in \Sigma\times M$. 
Let $\bar h\colon M\to [0,\infty]$ be the map defined by
$$
\bar h(x)=\int h(\omega,x)\, d\mathbb{P}(\omega).
$$
 Then for every $n\geq 1$ and $x_{0}\in M$ we have
$$
\mathbb{E}(h(F^{n}(\cdot, x_{0}))|\mathscr{F}_{n-1}) \geq C \bar h(f_{\omega}^{n}(x_{0}))
$$
for $\mathbb{P}$-almost every $\omega$.

\end{lemma}

\begin{proof}
By the definition of conditional expectations, we need to prove that for every measurable set $D\in \mathscr{F}_{n-1}$
\begin{equation}\label{nonintegrable}
\int_{D}h(\sigma^{n}(\omega),f_{\omega}(x_{0}))\,d\mathbb{P}(\omega)\geq\int_{D}\int h(\vartheta,
f_{\omega}^{n}(x_{0}))\, d\mathbb{P}(\vartheta)\, d\mathbb{P}(\omega).
\end{equation}
We start by proving that \eqref{nonintegrable} holds  for every characteristic function $h=
\mathds{1}_{A\times B}$, where  $A\times B\in \mathscr{F}\otimes \mathscr{B}$.
Note that $\mathds{1}_{A\times B}(\sigma^{n}(\omega),f_{\omega}(x_{0}))=
\mathds{1}_{A}(\sigma^{n}(\omega))\mathds{1}_{B}(f_{\omega}(x_{0}))$. Hence, we have 
\begin{equation}\label{cad1}
\int_{D}\mathds{1}_{A\times B}(\sigma^{n}(\omega),f_{\omega}(x_{0}))\,d\mathbb{P}(\omega)
=\int \mathds{1}_{A}(\sigma^{n}(\omega))\mathds{1}_{B}(f_{\omega}(x_{0}))
 \mathds{1}_{D}(\omega)\, d\mathbb{P}(\omega).
\end{equation}
Recall that if $u\colon \Sigma\to [0,\infty)$ is a measurable map and $g\colon  \Sigma\to [0,\infty)$ is a $\mathscr{F}_{n}$-measurable map, then the Markov property (see \cite[Proposition $2.13$]{Revuz}) implies that 
\begin{equation}\label{mpg}
\int (u\circ \sigma^{n})g\, d\mathbb{P}=\int \mathbb{E}_{\omega_{n}}(u) g(\omega)\, d\mathbb{P}(\omega).
\end{equation} 
Applying  \eqref{mpg} on the right hand side of \eqref{cad1} with $u=\mathds{1}_{A}$ and $g(\omega)=\mathds{1}_{B}(f_{\omega}(x_{0}))
 \mathds{1}_{D}(\omega)$ ($g$ is $\mathscr{F}_{n-1}$-measurable, and in particular, 
 it is $\mathscr{F}_{n}$-measurable), we obtain
 $$
 \int_{D}\mathds{1}_{A\times B}(\sigma^{n}(\omega),f_{\omega}(x_{0}))\,d\mathbb{P}(\omega)=
 \int \mathbb{E}_{\omega_{n}}(\mathds{1}_{A})\mathds{1}_{B}(f_{\omega}(x_{0}))
 \mathds{1}_{D}(\omega)\, d\mathbb{P}(\omega)\eqdef L.
 $$
 Since the mapping $\omega\mapsto \mathbb{E}_{\omega_{n}}(\mathds{1}_{A})\mathds{1}_{B}(f_{\omega}(x_{0}))
 \mathds{1}_{D}(\omega)$ is $\mathscr{F}_{n}$-measurable, it follows from the definition of the Markov measure
 $\mathbb{P}$ that
 $$
 L=
 \int \left(\int \mathbb{E}_{\omega_{n}}(\mathds{1}_{A})\mathds{1}_{B}(f_{\omega}(x_{0}))
 \mathds{1}_{D}(\omega) p(\omega_{n-1},d\omega_{n})\right) d\,\mathbb{P}(\omega).
 $$ 
 Since the mapping $\omega\mapsto \mathds{1}_{B}(f_{\omega}(x_{0}))
 \mathds{1}_{D}(\omega) $ does not depend on $\omega_{n}$, we have 
 $$
  L=\int \left(\mathds{1}_{B}(f_{\omega}(x_{0}))
 \mathds{1}_{D}(\omega) \int \mathbb{E}_{\omega_{n}}(\mathds{1}_{A})p(\omega_{n-1},d\omega_{n})\right) d\,\mathbb{P}(\omega).
 $$
Recalling that by hypothesis 
 $$
 C\leq \frac{dp_{\alpha}}{dm}\leq C^{-1}
 $$ for every $\alpha$, we conclude
 \[
\begin{split}
 L&\geq
 C\int  \left(\mathds{1}_{B}(f_{\omega}(x_{0}))
 \mathds{1}_{D}(\omega) \int \mathbb{E}_{\omega_{n}}(\mathds{1}_{A})
 \,dm(\omega_{n})\right) \,d\mathbb{P}(\omega)\\
 &=C\int  \mathds{1}_{B}(f_{\omega}(x_{0}))
 \mathds{1}_{D}(\omega) \mathbb{E}(\mathds{1}_{A})\, d\mathbb{P}(\omega)\\
 &=C\int \mathds{1}_{B}(f_{\omega}(x_{0}))
 \mathds{1}_{D}(\omega) \left(\int \mathds{1}_{A}(\vartheta)\, d\mathbb{P}(\vartheta)\right)\,d\mathbb{P}(\omega) \\
 &=C\int_{D}\int \mathds{1}_{A\times B}(\vartheta,f_{\omega}^{n}(x_{0}))\, d\mathbb{P}(\vartheta)\,d\mathbb{P}(\omega)=C\int_{D}\bar{h}(f^{n}_{\omega}(x_{0}))\,d\mathbb{P}(\omega).
\end{split}
\] 
This shows that \eqref{nonintegrable} holds for $h=\mathds{1}_{A\times B}$. 

  We now observe that if $R$ is an element of the algebra generated by the measurable rectangles, then the map $h=\mathds{1}_{R}$ satisfies \eqref{nonintegrable}.
   It is readily checked that the class  of measurable sets $R$ of $\Sigma\times M$ for which the map $h=\mathds{1}_{R}$ satisfies \eqref{nonintegrable} is a monotone class. Since the class of measurable rectangles generates 
the product $\sigma$-algebra of $\Sigma\times M$, it follows from  the monotone class 
theorem that if $R$ is a subset of $\mathscr{F}\otimes \mathscr{B}$, then the map $h=\mathds{1}_{R}$ satisfies \eqref{nonintegrable}. Now, by standard arguments of measure theory, it is easily verified that  every non-negative measurable map $\zeta$ satisfies \eqref{nonintegrable}.
This completes the proof of the  lemma.
\end{proof}

We also need  the following lemma. Let $\Pi_{2}\colon E\times M\to M$ be the projection on the second factor given by $\Pi_{2}(\alpha,x)=x$.

\begin{lemma}\label{equivalente} Let $\varphi$ be a bounded Markovian random iteration and let $\Xi$ be the bijection of Theorem \ref{one}. Given $\nu\in S(\varphi)$, let $\hat \mu$ denote $\Xi(\nu)$. Then, 
$$
C\hat \mu\leq \mathbb{P}\times \Pi_{2*}\nu\leq C^{-1} \hat \mu.
$$
\end{lemma}

\begin{proof}
We start by presenting a formula for the transition probability $q$. 
Recall that for every $\alpha\in E$, we have
$$
C\leq \frac{dp_{\alpha}}{dm}\leq C^{-1}.
$$
Set $k(\alpha,\beta)\eqdef \frac{dp_{\alpha}}{dm}(\beta)$. We claim that 
$q$ is given by 
\begin{equation}\label{cdual}
q(\beta,A)=\int_{A} k(\alpha,\beta)\, dm(\alpha).
\end{equation}
Indeed, 
\[
\begin{split}
\int_{A} q(\beta,B)\,dm(\beta)&=\int_{B}p(\alpha,A)\, dm(\alpha)\\
&=\int_{B}\left(\int_{A}\frac{dp_{\alpha}}{dm}(\beta)\,dm(\beta)\right) dm(\alpha)\\
&=\int_{A}\left(\int_{B}\frac{dp_{\alpha}}{dm}(\beta)\,dm(\alpha)\right) dm(\beta)
=\int_{A}\left(\int_{B}k(\alpha,\beta)\,dm(\alpha)\right)\, dm(\beta).
\end{split}
\]
This proves our claim.

Then, it follows from the characterization of $q$ in \eqref{cdual} that 
$$
C\leq \frac{dq_{\alpha}}{dm}\leq C^{-1}
$$
 for $m$-almost every $\alpha$, which implies that 
  $$
  C\leq \frac{dm}{dq_{\alpha}}\leq C^{-1}
  $$
   for $m$-almost every $\alpha$. Next, it is easily seen from the definition of the disintegration of measures that  
$\Pi_{2*}\nu=\int \nu_{\beta}\, dm(\beta)$. Hence,  for $\mathbb{P}$-almost every $\omega$  we have
$$
\Pi_{2*}\nu=\int \nu_{\beta}\frac{dm}{dq_{\omega_{0}}}\, q(\omega_{0},d\beta)\leq C^{-1}\int 
\nu_{\beta}\, q(\omega_{0},d\beta)=C^{-1}\mu_{\omega_{0}}
$$
and 
$$
\Pi_{2*}\nu=\int \nu_{\beta}\frac{dm}{dq_{\omega_{0}}}\, q(\omega_{0},d\beta)\geq C\int 
\nu_{\beta}\, q(\omega_{0},d\beta)=C \mu_{\omega_{0}}.
$$
In particular, for every measurable
rectangle $A\times B\in \mathscr{F}\otimes \mathscr{B}$, we have 
$$
\mathbb{P}\times \Pi_{2*}\nu(A\times B)=\int_{A}\Pi_{2*}\nu(B)\, d\mathbb{P}(\omega)\leq C^{-1}\int_{A} \mu_{\omega_{0}}(B)\, d\mathbb{P}(\omega)=C^{-1}\hat \mu(A\times B),
$$
and 
$$
\mathbb{P}\times \Pi_{2*}\nu(A\times B)=\int_{A}\Pi_{2*}\nu(B)\, d\mathbb{P}(\omega)\geq C\int_{A}  \mu_{\omega_{0}}(B)\, d\mathbb{P}(\omega)=C\hat \mu(A\times B),
$$
which is the desired result.
 
\end{proof}

\subsection{Proof of Theorem \ref{ergodic}}
Let $\Xi\colon S(\varphi)\to I_{0}(\varphi)$ be the bijection of Theorem \ref{one}. We need to prove that a stationary measure $\nu$ is ergodic if and only if the $\varphi$-invariant measure $\Xi(\nu)$ is ergodic.

We first assume that $\nu$ is ergodic. Let $\hat \mu$ 
denote $\Xi(\nu)$. To prove that $\hat \mu$ is ergodic, let $A$ be an $F$-invariant set (that is, $F^{-1}(A)=A$) and assume $\hat \mu(A)>0$. We claim that $\hat \mu(A)=1$. To prove this,  define 
 $h\colon \Sigma\times M\to \mathbb{R}$ by $h(\omega,x)=\mathds{1}_{A}(\omega,x)$ and note that for every $(\omega,x)\in \Sigma\times M$ we have
\begin{equation}\label{in}
F(h(\omega,x))=h(\omega,x).
\end{equation}
Let $\mathscr{F}_{n}$ be  the $\sigma$-algebra generated by the canonical projections 
$\omega\mapsto\omega_{i}$, $i=0,\dots, n$.
  From 
Levy's law we have that $\zeta(\omega,x_{0})=\lim_{n\to \infty} \mathbb{E}(\zeta(\cdot,x_{0}|\mathscr{F}_{n})(\omega)$
 for every $x_{0}\in M$ and $\mathbb{P}$-almost everywhere $\omega\in \Sigma$, and so we also have the convergence of the Cesaro averages
\begin{equation}\label{cesaro12}
h(\omega,x_{0})=\lim_{k\to \infty} \frac{1}{n}\sum_{i=0}^{n-1} \mathbb{E}(h(\cdot,x_{0}|\mathscr{F}_{i})(\omega)
\end{equation}
for every $x_{0}$ and $\mathbb{P}$-almost everywhere $\omega$.

Because of \eqref{in}, we have that $h(\omega,x_{0})= h(F^{i}(\omega,x_{0}))$
 for every $(\omega,x_{0})\in \Sigma\times M$ and $i\geq 1$. This implies that for every $i$
\begin{equation}\label{finiteint2}
\mathbb{E}(h(\cdot,x_{0}|\mathscr{F}_{i-1})(\omega)=\mathbb{E}(h(F^{i}(\cdot,x_{0}))|\mathscr{F}_{i-1})(\omega)
\end{equation}
for every $x_{0}$ and $\mathbb{P}$-almost every $\omega$.
We now apply Lemma \ref{difmarkocase} to the $F$-invariant map $h$ to obtain that for every $i$
 \begin{equation}\label{markovdesi}
\mathbb{E}(h(F^{i}(\cdot, x_{0}))|\mathscr{F}_{i-1})(\omega) \geq C \bar h(f_{\omega}^{i}(x_{0}))
\end{equation}
for every $x_{0}$ and $\mathbb{P}$-almost every $\omega$, where the map $\bar h\colon M\to \mathbb{R}$ is given by 
$$
\bar h(x)=\int h(\omega,x)\, d\mathbb{P}(\omega).
$$
Thus, it follows from \eqref{cesaro12}, \eqref{finiteint2} and \eqref{markovdesi} that 
\begin{equation}\label{arxiv3}
h(\omega,x_{0})\geq C \limsup_{k\to \infty} \frac{1}{n}\sum_{i=0}^{n-1}  \bar h(f_{\omega}^{i}(x_{0}))
\end{equation}
for every $x_{0}$ and $\mathbb{P}$-almost every $\omega$.
We need the following claim, which is just a direct corollary of the Birkhoff 
ergodic theorem for Markov chains. 
\begin{claim}
For $\mathbb{P}$-almost every $\omega$, we have
$$
\lim_{n\to \infty}\frac{1}{n}\sum_{i=0}^{n-1}\bar h(f^{i}_{\omega}(x))=\int \bar h\, d\Pi_{2*}\nu
$$
for $\Pi_{2*}\nu$-almost every $x$, where $\Pi_{2}$ is the projection on the second factor of $E\times M$.
\end{claim}
\begin{proof}
Define  $ H\colon E\times M\to \mathbb{R}$ by 
$
H(\alpha,x)=\bar h(x).
$
Recalling that for every $(\alpha,x)$ the sequence 
\[
\begin{cases}
Z_{0}=(\alpha,x)\\
Z_{n}=(\omega_{n-1},f_{\omega}^{n}(x))\quad \mbox{for} \quad n\geq 1
\end{cases}
\]
is a Markov chain 
with transition probability $\hat p$, then it follows from the  Birkhoff ergodic theorem  for Markov chains and the fact that $\nu$ is an ergodic stationary measure of $\hat p$ that, 
for $\mathbb{P}$-almost every $\omega$,
\begin{equation}\label{doesnot}
\lim_{n\to \infty}\frac{1}{n}\sum_{i=0}^{n-1}H(\omega_{i-1},f^{i}_{\omega}(x))=\int H\, d\nu
\end{equation}
for $\nu$-almost every $(\alpha,x)$. Since \eqref{doesnot} does not depend on $\alpha$, the definition of $H$ implies that,
for $\mathbb{P}$-almost every $\omega$, 
$$
\lim_{n\to \infty}\frac{1}{n}\sum_{i=0}^{n-1}\bar h(f^{i}_{\omega}(x))=\int \phi\, d\Pi_{2*}\nu
$$
for $\Pi_{2*}\nu$-almost every $x$.
This completes the proof of the claim. 
\end{proof}
Now, it follows from \eqref{arxiv3} and the previous claim   that
$$
h(\omega,x_{0})\geq C\int \bar h\, d\Pi_{2*}\nu=C\int h\, d(\mathbb{P}\times\Pi_{2*}\nu)
$$
for $\mathbb{P}\times \Pi_{2*}\nu$-almost every $(\omega,x_{0})$.
Next, Lemma \ref{equivalente} says that $ \mathbb{P}\times \Pi_{2*}\nu\geq C\hat \mu$, which implies that probability measure $\hat \mu$ is absolutely continuous with respect to $\mathbb{P}\times \Pi_{2*}\nu$. Therefore,
\begin{equation}\label{saysthat}
h(\omega,x_{0})\geq C^{2}\int h\, d\hat \mu
\end{equation}
for $\hat \mu$-almost every $(\omega,x_{0})$. Recalling that $h=\mathds{1}_{A}$ and $\mu(A)>0$, it follows from \eqref{saysthat} that 
$$
\mathds{1}_{A}(\omega,x_{0})\geq C^{2}\mu(A)>0
$$
for $\hat \mu$-almost every $(\omega,x_{0})\in \Sigma\times M$, and then we conclude that  $h(\omega,x_{0})=1$ for 
$\hat \mu$-almost every $(\omega,x_{0})$, which means that $\hat \mu(A)=1$. This proves that  $\hat \mu$ is ergodic.

We now assume that the probability measure $\hat \mu=\Xi(\nu)$ is ergodic. To see that $\nu$ is ergodic it is enough to prove that for $\nu$-almost every $(\alpha,x)$ 
the Markov chain $Z_{n}$ defined by 
\[
\begin{cases}
Z_{0}=(\alpha,x)\\
Z_{n}=(\omega_{n-1},f_{\omega}^{n}(x))\quad \mbox{for} \quad n\geq 1
\end{cases}
\]
satisfies 
$$
\lim_{n\to \infty} \frac{1}{n}\sum_{i=0}^{n}\phi(Z_{i})=\int \phi\, d\nu
$$
for $\mathbb{P}$-almost every $\omega$, which is equivalent to prove that 
$$
\lim_{n\to \infty} \frac{1}{n}\sum_{i=0}^{n}\phi((\omega_{i-1}, f_{\omega}^{i}(x))=\int \phi\, d\nu
$$
for $\mathbb{P}\times\Pi_{2*}\nu$-almost every $(\omega,x)$.
In order to show this, we
consider a bounded measurable map $\phi\colon E\times M\to \mathbb{R}$ and define  $\hat{\phi}\colon \Sigma\times M\to \mathbb{R}$ by 
$$
\hat{\phi}(\omega,x)=\phi(\omega_{0},f_{\omega_{0}}(x)).
$$
Since $\hat \mu$ is ergodic, it follows from the Birkhoff ergodic theorem 
that
\begin{equation}\label{talaaa}
\lim_{n\to \infty} \frac{1}{n}\sum_{i=0}^{n}\hat{\phi}(F^{i}(\omega,x))=\int \hat{\phi}\, d\hat\mu
\end{equation}
for $\hat \mu$-almost every $(\omega,x)$. We claim that 
$
\int \hat{\phi}\, d\hat \mu=\int \phi\,d\nu.
$
Indeed,
it follows from Theorem \ref{one} that   
$$
\nu_{\alpha}=f_{\alpha*}\mu_{\alpha}
$$
for $m$-almost every $\alpha$. Therefore
\[ 
\begin{split}
\int \hat{\phi}\, d\hat \mu&=\int \phi(\omega_{0},f_{\omega_{0}}(x))\, d\hat \mu_{\omega}(x)\,d\mathbb{P}(\omega)\\
&=\int \phi(\omega_{0},x)\, df_{\omega_{0}*} \mu_{\omega_{0}}(x)\, d\mathbb{P}(\omega)\\
&=\int \phi(\omega_{0},x)\,d\nu_{\omega_{0}}(x)\,dm(\omega_{0})=\int \phi\,d\nu,
\end{split}
\]
which proves our claim.

Since  $\hat{\phi}(F^{i}(\omega,x))=\phi(\omega_{i-1},f^{i}_{\omega}(x))$ for every $i$,
  we conclude from \eqref{talaaa} that 
\begin{equation}\label{fiii}
\lim_{n\to \infty} \frac{1}{n}\sum_{i=0}^{n-1}\phi(\omega_{i-1},f_{\omega}^{i}(x))=\int \phi\, d\nu
\end{equation}
for $\hat \mu$-almost every $(\omega,x)$. Now, Lemma \ref{equivalente} implies that $\mathbb{P}\times\Pi_{2*}\nu$ is absolutely continuous with respect to $\hat \mu$. Therefore, from \eqref{fiii} we conclude that
$$
\lim_{n\to \infty} \frac{1}{n}\sum_{i=0}^{n}\phi((\omega_{i-1},f_{\omega}^{i}(x))=\int \phi\, d\nu
$$
for $\mathbb{P}\times \Pi_{2*}\nu$-almost every $(\omega,x)$, which implies that $\nu$ is ergodic.

This completes the proof of the theorem.
\hfill \qed
\section{Local synchronization}\label{sectionlocsinc}
In this section, we prove Theorem \ref{teocontraction}.
We start with a preliminary technical result. In what follows in this section, $p$ is a transition probability on a compact metric space $E$ (endowed with its Borel $\sigma$-algebra) having a unique stationary measure $m$, the pair $(p,m)$ is bounded, and $C$ denotes a constant satisfying \eqref{faci}. We also assume that
the map $\alpha\mapsto p(\alpha,\cdot)$ is continuous in the weak$\star$-topology.

\begin{theorem}\label{upperr}
Let $M$ be a compact metric space and consider  a measurable map $f\colon E\times M\to M$ such that $f_{\alpha}$ is continuous for every $\alpha\in E$. 
Let $F\colon \Sigma\times M\to \Sigma\times M$ be the skew product induced by the Markovian random iteration  $\varphi$ associated with $(p,m)$ and the map $f$. 
Let $\zeta\colon \Sigma\times M\to [0,\infty]$ be a measurable map such that:
\begin{enumerate}
\item[(i)]
For every $\omega\in \Sigma$, the map $x\mapsto\zeta(\omega,x)$ is lower semi-continuous.

\item[(i)] $
\zeta(\omega,x)\geq \zeta(F(\omega,x))$ for every $(\omega,x)\in \Sigma\times M$.
\end{enumerate}
Then for every $x_{0}\in M$, for $\mathbb{P}$-almost every $\omega$ there is a subset 
$U_{\omega,x_{0}}$ of  $I_{0}(\varphi)$ such that  
$$
\zeta(\omega,x_{0})\geq C^{2} \sup_{\hat \mu\in U_{\omega,x_{0}}} 
\int \zeta(\vartheta,x)\, d\hat \mu(\vartheta,x).
$$ 
\end{theorem}

\begin{remark}\emph{
In this paper, we will use Theorem \ref{upperr} only for the case where $M=S^{1}$  and $\zeta$ is $F$-invariant, that is, $\zeta(\omega,x)= \zeta(F(\omega,x))$ for every $(\omega,x)\in \Sigma\times M$. However, we chose to state Theorem \ref{upperr} in such 
generality because the proof is the same in any case and we believe that it can be a useful tool in the study of Markovian 
random iterations. See \cite[Lemma 3.20]{Malicet} for the i.i.d. version of Theorem \ref{upperr}.}
\end{remark}

\begin{proof}[Proof of Theorem \ref{upperr}]

Denote by $A_{\omega,x_{0}}$ the set of accumulations points of the sequence of probability measures  
$\frac{1}{n}\sum_{i=0}^{n-1}\delta_{f_{\omega}^{i}(x_{0})}$.
We need the following estimative:
\begin{proposition}\label{firstestimative}
For every $x_{0}\in M$, for $\mathbb{P}$-almost every $\omega$ we have
$$
\zeta(\omega,x_{0})\geq C \sup_{\mu\in A_{\omega,x_{0}}} \int \int \zeta(\vartheta,x)\,d\mathbb{P}(\vartheta)d\mu(x).
$$
\end{proposition}
\begin{proof}
 Given $\mu\in  A_{\omega,x_{0}}$, by definition there is a subsequence $n_{k}$ such that 
$$
\lim_{k\to \infty} \frac{1}{n_{k}}\sum_{i=0}^{n_{k}-1}\delta_{f_{\omega}^{i}(x_{0})}=\mu.
$$
Let $\mathscr{F}_{n}$ be  the $\sigma$-algebra generated by the canonical projections 
$\omega\mapsto\omega_{i}$, $i=0,\dots, n$.
  From 
Levy's law we have that $\zeta(\omega,x_{0})=\lim_{n\to \infty} \mathbb{E}(\zeta(\cdot,x_{0}|\mathscr{F}_{n})(\omega)$
for $\mathbb{P}$-almost everywhere $\omega\in \Sigma$. In particular, 
\begin{equation}\label{cesaro1}
\zeta(\omega,x_{0})=\lim_{k\to \infty} \frac{1}{n_{k}}\sum_{i=0}^{n_{k}-1} \mathbb{E}(\zeta(\cdot,x_{0}|\mathscr{F}_{i})(\omega)
\end{equation}
for $\mathbb{P}$-almost everywhere $\omega$.
Since $\zeta(\omega,x)\geq \zeta(F(\omega,x))$ for every $(\omega,x)\in \Sigma\times M$, we have  $\zeta(\omega,x_{0})\geq\zeta(F^{i}(\omega,x_{0}))$ 
for every $\omega$ and $i$, which implies that 
\begin{equation}\label{finiteint}
\mathbb{E}(\zeta(\cdot,x_{0}|\mathscr{F}_{i-1})(\omega)\geq\mathbb{E}(\zeta(F^{i}(\cdot,x_{0}))|\mathscr{F}_{i-1})(\omega).
\end{equation}

Then, it follows from \eqref{finiteint}, \eqref{cesaro1} and Lemma \ref{difmarkocase} that 
for $\mathbb{P}$-almost every $\omega$, we have
\[
\begin{split}
\zeta(\omega,x_{0})&\geq C \liminf_{k\to \infty} \frac{1}{n_{k}}\sum_{i=0}^{n_{k}-1} \bar \zeta(f_{\omega}^{n}(x_{0}))\\
&=C\liminf_{k\to \infty} \int \bar\zeta \, d\left(  \frac{1}{n_{k}}\sum_{i=0}^{n_{k}-1} \delta_{f_{\omega}^{n}(x_{0})} \right).
\end{split}
\]
Next, because $\zeta$ is lower semi-continuous,  the map $\bar \zeta$ is also 
lower semi-continuous. Hence, it follows from Portmanteau theorem that 
$$
\zeta(\omega,x_{0})\geq C \int \bar \zeta \, d\mu=C\int \int \zeta(\vartheta,x)\,d\mathbb{P}(\vartheta)d\mu(x)
$$
for $\mathbb{P}$-almost every $\omega$.
This completes the proof of the proposition. 
\end{proof}

Fix $x_{0}\in M$. Let $\hat p$ be the transition probability of the Markov chain  $Z_{n}(\omega)=(\omega_{n-1},f_{\omega}^{i}(x_{0}))\in \Sigma\times M$, recall the definition in \eqref{transition}.
 Let $S_{\omega,x_{0}}$ be the set 
of accumulation points of the sequence of probability measures 
$$
\frac{1}{n}\sum_{i=0}^{n-1}\delta_{Z_{n}(\omega)}.
$$
By hypothesis, $f_{\alpha}$ is continuous for every $\alpha$ and the mapping $\alpha\mapsto p(\alpha,\cdot)$ is continuous in the weak$\star$-topology. Hence, it is readily checked that the map 
 $$
 (\alpha,x)\mapsto \hat p((\alpha,x),\cdot)
 $$
 is also continuous in the weak$\star$-topology. Since $E\times M$ is a compact metric space,  it follows from  a result on general Markov chains  due  to Furstenberg, see  \cite[Lemma $7.1$]{Fur}, that there is a subset $\Sigma_{0}\subset\Sigma$ of $\mathbb{P}$-full measure 
such that for every $\omega\in \Sigma_{0}$ the set $S_{\omega,x_{0}}$ is constituted by stationary measures of the Markovian random iteration $\varphi$, that is, $S_{\omega,x_{0}}\subset S(\varphi)$. Let $\Xi\colon S(\varphi)\to I_{0}(\varphi)$ be the bijection of Theorem \ref{one}. 
Then, for every $\omega\in \Sigma_{0}$, we define  
 $$
 U_{\omega,x_{0}}\eqdef \Xi(S_{\omega,x_{0}}).
  $$

We are now ready to  prove  Theorem \ref{upperr}. To this end, let $\hat \mu\in U_{\omega,x_{0}}$. By definition, there is a stationary measure
 $\nu \in S_{\omega,x_{0}}$ such that $\Xi(\nu)=\hat \mu$.
On the other hand, by definition of $S_{\omega,x_{0}}$,
there is a subsequence $(n_{k})_{k}$ 
 such that 
$$
 \nu=\lim_{n\to \infty}\frac{1}{n_{k}}\sum_{i=0}^{n_{k}-1}\delta_{(\omega_{i-1},f_{\omega}^{i}(x_{0}))}.
$$ 
Let $\Pi_{2}$ be the projection on the second factor of $E\times M$.
We observe that  $\Pi_{2*}\nu\in A_{\omega,x_{0}}$, and hence we can  apply  Proposition \ref{firstestimative} to obtain 
 \begin{equation}\label{lateral}
  \zeta(\omega,x_{0})\geq C\iint \zeta(\vartheta,x) \, d\Pi_{2*}\nu(x) d\mathbb{P}(\vartheta).
\end{equation}
We now  apply Lemma \ref{equivalente}  to equation \eqref{lateral} to obtain 
\[
\begin{split}
\zeta(\omega,x_{0})
 &\geq C^{2} \int \zeta(\vartheta,x) \, d\hat \mu(\vartheta,x).
\end{split}
\]
The proof of the theorem is now complete.

\end{proof}
\subsection{Reformulation of Theorem \ref{teocontraction}}
We now state a reformulation of Theorem \ref{teocontraction} using the  \emph{exponent 
of contraction} of a random iteration  of homeomorphisms
introduced in \cite{Malicet}. 

Let $\varphi$ be a Markovian random iteration of homeomorphisms of the circle $S^{1}$, as in the statement of Theorem \ref{teocontraction}. The \emph{exponent of 
contraction of $\varphi$} at the point $(\omega,x)\in \Sigma\times S^{1}$ is the non-positive quantity
$$
\lambda(\omega,x)=\varlimsup_{y\to x}\, \varlimsup_{n\to \infty} \frac{\log (Leb[f^{n}_{\omega}(y),f^{n}_{\omega}(x)])}{n}
$$
where $Leb$ denotes the Lebesgue measure on $S^{1}$ and $[a,b]$ denotes the arc of smaller length determined by the points $a$ and $b$.
\begin{theorem}\label{reformulation}
Under assumptions of Theorem \ref{teocontraction} there is $\lambda_{0}<0$ such that for every 
$x\in S^{1}$ 
 we have 
$$
\lambda(\omega,x)\leq \lambda_{0}
$$  
for $\mathbb{P}$-almost every $\omega$ in $\Sigma$.
\end{theorem}\label{ivp}
The proof of Theorem \ref{reformulation} is a consequence of Theorem \ref{upperr} combined with the invariance principle \cite[Theorem F]{Malicet}.  To state this invariance principle for $\varphi$,
we need the following definition. Let  $\hat \mu$ be a $\varphi$-invariant measure. The \emph{exponent of contraction 
of $\hat \mu$} is the non-positive quantity 
$$
\lambda(\hat\mu)=\int\lambda(\omega,x) \, d\hat \mu(\omega,x).
$$
\begin{theorem}[Invariance principle]\label{ivp}
 Let $\hat \mu$ be a $\varphi$-invariant measure such that $\lambda(\hat \mu)=0$. Then 
$$
f_{\omega_{0}*}\hat \mu_{\omega}=\hat \mu_{\sigma(\omega)}
$$
for $\mathbb{P}$-almost every $\omega\in \Sigma$.
\end{theorem}
The theorem above is just a particular case of a   general result proved by Malicet in \cite[Theorem F]{Malicet}  for random iterations of homeomorphisms of $S^{1}$. We observe that the invariance principle for random iterations of homeomorphisms of $S^{1}$ has the same flavor of the classical invariance principle of Ledrappier \cite{Led} for random products
 of matrices.

\begin{proof}[Proof of Theorem \ref{reformulation}]

  We start by showing that for every $\hat \mu\in \mathcal{I}_{0}(\varphi) $ we have  $\lambda(\hat \mu)<0$. Indeed, assume that $\lambda(\hat \mu)=0$. Applying Theorem \ref{ivp}  we obtain
\begin{equation}\label{1m}
f_{\omega_{0}*}\hat \mu_{\omega}=\hat \mu_{\sigma(\omega)}
\end{equation}
for $\mathbb{P}$-almost every $\omega$. Let $\{\mu_{\alpha}\colon \alpha\in E\}$ be the family of probability measures for which $\hat \mu_{\omega}=\mu_{\omega_{0}}$. Hence, it follows from \eqref{1m} that 
$$
f_{\omega_{0}*}\mu_{\omega_{0}}=\mu_{\omega_{1}}
$$
for $\mathbb{P}$-almost every $\omega$,
or equivalently, for $m$-almost every $\alpha$ we have
\begin{equation}\label{ultimaaaa}
f_{\alpha*}\mu_{\alpha}=\mu_{\beta}
\end{equation}
for $p_{\alpha}$-almost every $\beta$.
Because $m\ll p_{\alpha}$ for $m$-almost every $\alpha$, we get from \eqref{ultimaaaa} that $\beta\mapsto\mu_{\beta}$ is constant ($m$-a.e.). This shows that the maps $f_{\alpha}$ have an invariant measure in common, which is a contradiction. Therefore, the exponent of contraction $\lambda(\hat \mu)$ is negative. 

We now observe that $\lambda$ is an $F$-invariant map and for every $\omega$, the map $x\mapsto \lambda(\omega,x)$ is upper semi-continuous. In particular, $-\lambda$ is $F$-invariant and for every $\omega$  the map $x\mapsto -\lambda(\omega,x)$ is
lower semi-continuous. Hence, all conditions required in Theorem \ref{upperr} are met for the map $-\lambda$  and so we can apply this theorem  to get that for 
every $x\in S^{1}$, for $\mathbb{P}$-almost every $\omega$, we have
$$
\lambda(\omega,x)\leq \inf_{\hat \mu\in U_{\omega,x}} C^{2}\lambda(\hat \mu)
$$
where $U_{\omega,x}$ is a subset of $\mathcal{I}_{0}(\varphi)$. Since  $\lambda(\hat \mu)<0$ for every $\hat \mu\in\mathcal{I}_{0}(\varphi)$, we conclude that for every  $x\in S^{1}$, for $\mathbb{P}$-almost every $\omega$, we have
\begin{equation} \label{because}
\lambda(\omega,x)< 0.
\end{equation}
It remains to see the existence of a uniform bound as claimed in Theorem \ref{reformulation}. 
To this end, we apply Proposition \ref{firstestimative}
to the map $-\lambda$ to obtain that
for every $x_{0}\in M$, for $\mathbb{P}$-almost every $\omega$, we have
\begin{equation}\label{777}
\lambda(\omega,x_{0})\leq C \inf_{\mu\in A_{\omega,x_{0}}} \int \int \lambda(\vartheta,x)\,d\mathbb{P}(\vartheta)d\mu(x).
\end{equation}
We now use the fact that $\lambda$ is upper semi-continuous on the second variable to conclude that the map $\bar \lambda\colon S^{1}\to \mathbb{R}$ defined by 
$$
\bar\lambda(x)= \int \lambda(\omega,x)\, d\mathbb{P}(\omega)
$$ is also upper semi-continuous. Because of \eqref{because}, we have that $
\bar\lambda(x)<0
$
for every $x\in S^{1}$.
  Since any upper semi-continuous map on a compact metric space has a
maximum value, we conclude that there is $\lambda_{0}<0$ such that 
$$
\bar\lambda(x)\leq \lambda_{0}
$$ 
for every $x\in S^{1}$. 
Therefore, because $C\leq 1$, it follows from \eqref{777}
 that for every $x_{0}$, for $\mathbb{P}$-almost every $\omega$, we have
 $$
 \lambda(\omega,x_{0})\leq  \inf_{\mu\in A_{\omega,x_{0}}} \int \bar \lambda(x)d\mu(x)\leq \lambda_{0},
 $$
 which is the desired result.
\end{proof}
\subsection{Proof of Theorem \ref{teocontraction}}
 Theorem \ref{teocontraction} follows directly from the definition of limit superior  and Theorem \ref{reformulation}.

\hfill \qed

\section{Related and future work}
A natural question would be  whether we can obtain generalizations of  Theorem \ref{teocontraction} for non-Markovian random iterations of homeomorphisms of  $S^{1}$.
This question leads us to a general question on ergodic theory. Indeed, recall that
we have proved  Theorem \ref{teocontraction} using a generalization of the Breiman ergodic theorem \cite{Breiman} for Markov chains obtained by Furstenberg in \cite[Lemma 7.1]{Fur}. This result is a kind of Krylov-Bogolyubov theorem for Markov chains and stationary measures.   Thus, a good start to generalize Theorem \ref{teocontraction} would be  to obtain a version for skew products of this Breiman-Furstenberg theorem.

To be more precise, we need some definitions. Let $(\Omega,\mathscr{F},\mathbb{P},\theta)$ be a measure preserving dynamical system. Consider a measurable space $M$ and let $f\colon \Omega\times M\to M$ be a measurable map.  Denote by $f_{\omega}$ the map defined by $f_{\omega}(x)=f(\omega,x)$.  Then, the map $\varphi\colon \mathbb{N}\times \Omega\times M \to M$ defined by 
$$
\varphi(n,\omega,x)=f_{\theta^{n-1}(\omega)}\circ \dots\circ  f_{\omega}(x)
$$
is called a \emph{random iteration of maps}. A probability measure $\mu$ on $\Omega\times M$
is called $\varphi$-\emph{invariant} if $\mu$ has first marginal $\mathbb{P}$ and is invariant 
by the skew product map $F$ defined by 
$$
F(\omega,x)=(\theta(\omega),f_{\omega}(x)).
$$ 
The \emph{past} of a random iteration is the $\sigma$-algebra defined by
$$
\mathscr{F}^{-}\eqdef\sigma(\omega\mapsto f_{\theta^{-n}(\omega)}^{n}(x)\colon n\geq 1 , x\in X),
$$ 
that is, $\mathscr{F}^{-}$ is the smallest $\sigma$-algebra that makes
all maps in the above family measurable, see \cite{Crauel} for details. We will say that  the disintegration $\{\mu_{\omega}\colon \omega\in \Omega\}$ of a $\varphi$-invariant measure $\mu$ with respect to $\mathbb{P}$  depends only on the past if for every measurable set $A$ the map $\omega\mapsto \mu_{\omega}(A)$ is $\mathscr{F}^{-}$-measurable.

We recall that if $\Omega$ is a product space $E^{\mathbb{Z}}$, $\theta$ is the shift map and 
$\mathbb{P}$ is a product measure $\nu^{\mathbb{Z}}$, then there is a one-to-one correspondence between the set of stationary measures and the set of $\varphi$-invariant measures whose disintegration  depends only on the past, see Arnold \cite[Theorem 2.1.8]{Arnold}. 
Note that Theorem \ref{one} combined with \cite[Theorem 1.7.2]{Arnold} implies that there is 
such a correspondence also in the Markovian case.
 
 Thus, when we can not consider stationary measures, it seems natural to take the set of $\varphi$-invariant measures whose disintegration depends only on the past to generalize results on which stationary measures play a role.
  Having these remarks in mind, we ask the following:
 
\begin{question*}
For every $x$, for $\mathbb{P}$-almost every $\omega$,  is every accumulation point of 
$$
\frac{1}{n}\sum_{i=0}^{n-1}\delta_{F^{i}(\omega,x)}
$$
 a $\varphi$-invariant measure  whose disintegration depends only on the past?
\end{question*}

\bibliographystyle{acm}
\bibliography{references}

\end{document}